\documentclass[12pt]{amsart}
\date{24 Nov, 2020}
\usepackage{latexsym,amsmath,amsfonts,amscd,amssymb}
\usepackage{stmaryrd}
\usepackage{lscape}
\usepackage{array}   
\usepackage[mathscr]{eucal} 
\usepackage{mathrsfs}
\usepackage{graphicx}
\usepackage{xypic}
\usepackage{tikz-cd}
\usetikzlibrary{positioning, matrix, arrows}
\usepackage{calligra}
\usepackage[T1]{fontenc}
\DeclareFontFamily{OT1}{pzc}{}
\DeclareFontShape{OT1}{pzc}{m}{it}{<-> s * [1.10] pzcmi7t}{}
\DeclareMathAlphabet{\mathpzc}{OT1}{pzc}{m}{it}

\DeclareFontFamily{OT1}{rsfs}{}
 \DeclareFontShape{OT1}{rsfs}{n}{it}{<->rsfs10}{}
 \DeclareMathAlphabet{\curly}{OT1}{rsfs}{n}{it}

\theoremstyle{plain}  
\newtheorem{theorem}{Theorem}[section]

\newtheorem*{theorem*}{Theorem}

\newtheorem{lemma}[theorem]{Lemma}

\theoremstyle{remark}

\newtheorem*{claim*}{Claim}
\numberwithin{equation}{section}

\renewcommand{\leq}{\leqslant}

\renewcommand{\geq}{\geqslant}
\renewcommand{\ge}{\geqslant}

\renewcommand{\phi}{\varphi}

\renewcommand{\phi}{\varphi}

\begin{document}
\title[Stability of Syzygy bundles ]{Stability of Syzygy bundles corresponding to stable vector bundles on algebraic surfaces}
\author{SURATNO BASU}
\email{suratno@cmi.ac.in, suratno.math@gmail.com}
\address{Chennai Mathematical Institute, Siruseri Sipcot IT Park,
         Kelambakkam, Chennai-603103, Tamilnadu, India}
\author{SARBESWAR PAL}
\email{sarbeswar11@gmail.com, spal@iisertvm.ac.in}
\address{IISER Thiruvananthapuram, Maruthamala P. O., Kerala 695551}


\begin{abstract}
 Let $(X, H)$ be a polarized smooth projective algebraic surface and $E$ is globally generated, stable vector bundle on $X$. Then  the Syzygy
 bundle $M_E$ associated to it is defined as the kernel bundle corresponding to the evaluation map. In this article we will study the stability property of $M_E$ 
 with respect to $H$.
\end{abstract}

\maketitle

\section{Introduction}


The purpose of this paper is to investigate the stability of Syzygy bundles associated to a stable 
and globally generated vector bundles on a smooth projective algebraic surface.

Let $X$ be a smooth, irreducible, projective algebraic variety defined over an algebraically closed field $k$. 
We fix a very ample divisor $H$ on $X$. We refer to the pair $(X,H)$ as a polarized algebraic variety. Let 
$E$ be a globally generated vector bundle on $X$. Then the Syzygy bundle $M_{_E}$ is defined to be the kernel 
bundle corresponding to the evaluation map $\text{ev}:H^0(E)\otimes \mathcal{O}_{_X}\to E\to 0$. Thus we have an exact sequence 
\[0\to M_{_E}\to H^0(E)\otimes \mathcal{O}_{_X}\to E\to 0.\]

These vector bundles (and some analogues) arise in a variety of geometric and algebraic problems. 
For example  what are the numerical conditions (optimal) on vector bundles $E$ and $F$ on $X$ such that the natural  product map
\[
H^0(X, E) \times H^0(X, F) \to H^0(X, E \otimes F),
\]
becomes surjective.  The first initiative towards this question  was taken by D.C. Butler \cite{B}. 
He gave an affirmative answer when $X$ is a smooth projective curve. He approached the question  via the vector
bundle $M_E$
 associated to a bundle $E$- generated by global sections. 
 Consequently, there has been considerable interest in trying to establish the stability of
$M_E$ in various settings. When $X$ is a smooth curve of genus $g\ge  1$, the semistability was studied by Butler. 
He proved that when $E$ is semistable with $\mu(E) \ge 2g$, then $M_E$ is semistable.

Our main aim in this paper is to study the (slope) stability of $M_{_E}$ with respect to $H$ when $E$ is stable with respect to $H$ and $X$ is a smooth projective 
surface.
In \cite{ELM} the authors study the stability of $M_{_E}$ when $E$ is a very ample line bundle. In fact they showed that, if we take sufficiently 
large power of $E$ then the kernel bundle is (slope) stable with respect to $E$. In this short note we consider the case of $H$-stable vector bundles.
For any vector bundle $E$ on $X$ and $m>0$ let $E(m):=E\otimes \mathcal{O}_{_X}(mH)$.
We will prove the following

\begin{theorem}
 Let $E$ be a $H$-(slope) stable vector bundle on $X$.
 There exists $m>>0$ such that the kernel bundle $M_{_{E(m)}}$ is stable with respect to $H$.
\end{theorem}

We follow the main proof strategy of \cite{ELM}. Our method is suitable enhancement of the arguments given in \cite{ELM}.
Two key 
ingredients of the proof are Mehta-Ramanathan restriction theorem and Butler's theorem on stability of kernel bundles on curves.
We note that our theorem as well as the theorem of \cite{ELM} are not effective in a sense there is no concrete lower bound of $m$.

\section{Proof Of the main theorem}
This section is devoted to the proof of the main theorem. First let us fix up some notations.
As before, $X$ is a smooth, irreducible projective surface and $H$ a very ample divisor. Let $E$ be a stable vector bundle with respect to $H$ of rank $l$ 
on $X$.
For the entire course of arguments we fix an integer $n>>0$, sufficiently large, such that 
$nH-K_{_X}$ is very ample where $K_{_X}$ is the canonical divisor and $H^i(X,E(n))=0$, $i=1,2$. For any closed point $x$ let $m_{_x}$ denotes the ideal defining the 
the point $x$.

We observe the following easy Lemma.
\begin{lemma}
If $W\subset H^0(X,E(n))$ is a subspace which generates $E(n)$ then the natural multiplication 
map \[H^0(X,\mathcal{O}_{_X}((m-n)H)\otimes m_{_x}))\otimes W\to H^0(X,E(m)\otimes m_{_x})\] is surjective for $m > > 0$.
\end{lemma}

\begin{proof}

Note that it is enough to prove that the map 
\[H^0(X,\mathcal{O}_{_X}((m-n)H)))\otimes W \to H^0(X,E(m))\] is surjective for some $m>>0$.
By definition of $W$ we have a surjection $W\otimes \mathcal{O}_{_X}(-n)\to E$. Let $K$ be the kernel 
of this surjection. Then we have the following exact sequence 
\[0\to K\to W\otimes \mathcal{O}_{_X}(-n)\to E\to 0\]
We choose $m>>0$ such that $H^1(X,K(m))=0$. After tensoring the above exact sequence by $\mathcal{O}_{_X}(m)$ and passing to the corresponding long 
exact sequence we get \[H^0(X,\mathcal{O}_{_X}((m-n)H)))\otimes W \to H^0(X,E(m))\] is surjective.

\end{proof}
We now choose a $m>>0$ independent of $n$ 
such that 

(1) $\mathcal{O}_{_X}(mH)$ and $\mathcal{O}((m-n)H)$ are very ample.

(2) If $W\subset H^0(X,E(n))$ is a subspace which generates $E(n)$ then the natural multiplication 
map \[H^0(X,\mathcal{O}(m-n)H\otimes m_{_x})\otimes W\to H^0(X,E(m)\otimes m_{_x})\] is surjective.


    
 We aim to show that there exists a $m>>0$ such that $M_{_{E(m)}}$ is $H$-stable. We first analyze if for some $m>0$, 
 $M_{_{E(m)}}$ is not $H$-stable. In this case there exists a saturated locally free subsheaf $F_{_m}\subset M_{_{E(m)}}$ 
 such that \[\frac{c_{_1}(F_{_m}).H}{rk(F_{_m})}\geq \frac{c_{_1}(M_{_{E(m)}}).H}{rk(M_{_{E(m)}})}\]
 Our goal is to show that for sufficiently large $m>>0$ no such $F_{_m}$ can exists. 
 
 Pick a smooth and irreducible curve $C_{_m}\in |(m-n)H|$ through a fixed point $x\in X$. We may also assume that 
 $M_{_{E(m)}}/F_{_m}$ is locally free along $C_{_m}$. Observe that 
 \[\mu_{_H}(F_{_m})=\frac{c_{_1}(F_{_m}).H}{rk(F_{_m})}=\frac{1}{(m-n)}\mu(F_{_m}|_{_{C_m}}),\]
 Similarly, \[\mu_{_H}(M_{_{E(m)}})=\frac{1}{(m-n)}\mu(M_{_{E(m)}}|_{_{C_{_m}}})\]
 
 Thus we have \begin{equation}\label{eqn3}
               \mu(F_{_m}|_{_{C_m}})\geq \mu(M_{_{E(m)}}|_{_{C_{_m}}})
              \end{equation}

Since, $H^1(X,E(n))=0$ we have the following exact sequence 
\begin{equation}\label{eqn4}
 0\to \mathcal{O}_{_C}^{h^0(E(n))}\to M_{_{E(m)}}|_{_{C_{_m}}}\to \overline{M}_{_m}\to 0
\end{equation}
where $\overline{M}_{_m}$ is the kernel bundle corresponding to $E(m)|_{_{C_m}}$. 

As $E$ is stable with respect to $H$,  $E(m)$ is stable with respect to $H$. Now if we choose $m$ so that $(m-n)>>0$ the 
by Mehta-Ramanathan restriction theorem (\cite[Theorem 7.2.8]{HL}) $E(m)|_{_{C_m}}$ is stable on $C_{_m}$. Now we have 
\[\text{deg}(E(m)|_{_{C_{_m}}})=\text{deg}(E|_{_{C_m}})+r\text{deg}(\mathcal{O}_{_X}(m)|_{_{C_m}})\]
Write 
\[mH=K_{_X}+mH-nH+Q\] where $Q=nH-K_{_X}$. By our assumption $Q$ is very ample. Since, $C_{m}\in |(m-n)H|$ by adjunction formula 
we get $K_{_{C_m}}=(K_{_X}+(m-n)H)|_{_{C_m}}$. Thus $\text{deg}(\mathcal{O}_{_X}(m)|_{_{C_m}})=\text{deg}(K_{_{C_m}})+Q.C_{_m}$.
As $Q.C_{_m}\geq 3$ 
\begin{equation}\label{eqn5}
\text{deg}(K_{_{C_m}})+Q.C_{_m}\geq 2g_{_{C_m}}+1
 \end{equation}

We have 
\[\text{deg}(E(m)|_{_{C_{_m}}})=\text{deg}(E|_{_{C_m}})+r\text{deg}(\mathcal{O}_{_X}(m)|_{_{C_m}})\geq \text{deg}(E|_{_{C_m}})+r(2g_{_{C_m}}+1)\]
We assume that $c_1(E)$ is effective then $c_1(E).C_{_m}\geq 0$. Then we have 
\[\text{deg}(E|_{_{C_m}})+r(2g_{_{C_m}}+1)\geq r(2g_{_{C_m}}+1)\]
Therefore, $\text{deg}(E(m)|_{_{C_m}})\geq r(2g_{_{C_m}}+1)$ and hence 
$\mu(E(m)|_{_{C_m}})\geq (2g_{_{C_m}}+1)$. By Butler's theorem (\cite[Theorem 1.2]{B}) we have $\overline{M}_{_m}$ is stable. 

Let $K_{_m}$ be the kernel of $F_{_m}\hookrightarrow M_{_{E(m)}}|_{_{C_{_m}}}\to \overline{M}_{_m}$ and $N_{_m}$ be the image. As $\overline{M}_{_m}$ is 
stable $K_{_m}\neq 0$.
Then we have the following commutative diagram
\begin{equation}\label{eqn6} \begin{tikzcd}[arrows={-Stealth}]
 0\rar & K_m\rar \dar & F_m\rar \dar  & N_{_m}\rar \dar  & 0\\%
0\rar & \mathcal{O}_{_{C_m}}^{h^0(E(n))}\rar & M_{_{E(m)}}|_{_{C_{_m}}}\rar & \overline{M}_{_m}\rar & 0
\end{tikzcd}
\end{equation}

We complete the proof following two crucial lemmas. The proofs of these two lemmas follow the exact line of arguments of \cite[Lemma 1.1, Lemma 1.2]{ELM}
with some modifications. However, for completeness sake we will provide the proofs.
\begin{lemma}\label{rk1}
 For any $x\in X$, $\text{rank}(F_{_m})\geq h^0(X,\mathcal{O}_{_X}((m-n)H)\otimes m_{_x})=h^0(X,\mathcal{O}_{_X}((m-n)H))-1$.
 \begin{proof}
  For any vector space $V$ on $X$, $\mathbb{P}_{_{\text{sub}}}(V)$ denotes the projective space associated to $V$ 
  consisting $1$ dimensional linear subspace of $V$. Multiplication of sections gives rise to a finite morphism 
 
 \[\mu_{_m}:\mathbb{P}_{_{\text{sub}}}(H^0(X,\mathcal{O}_{_X}((m-n)H)\otimes m_{_x})\times \mathbb{P}_{_{\text{sub}}}(H^0(X,E(n)))\to \mathbb{P}_{_{\text{sub}}}(H^0(X,E(m)\otimes m_x))\]
  which sends \[([s_1],[s_2])\mapsto s_1\otimes s_2.\]
  Note that this morphism is composition of the Segre embedding 
  \[\mathbb{P}_{_{\text{sub}}}(H^0(X,\mathcal{O}_{_X}((m-n)H)\otimes m_{_x})\times \mathbb{P}_{_{\text{sub}}}(H^0(X,E(n)))\hookrightarrow \mathbb{P}_{_{\text{sub}}}(H^0((X,\mathcal{O}_{_X}((m-n)H)\otimes m_{_x}))\otimes H^0(X,E(n)),\]
  followed by the rational map 
  \[\mathbb{P}_{_{\text{sub}}}(H^0((X,\mathcal{O}_{_X}((m-n)H)\otimes m_{_x}))\otimes H^0(X,E(n))\dashrightarrow \mathbb{P}_{_{\text{sub}}}(H^0(X,E(m)\otimes m_x))\]
  For any $x\in X$ we have $F_{_m}(x)\subseteq M_{_m}(x)=H^0(X,E(m)\otimes m_{_x})$. 
  Let $Z:=\mu_{_m}^{-1}(\mathbb{P}_{_{\text{sub}}}(F_{_m}(x))$. Then $\mu_{_m}:Z\to \mathbb{P}_{_{\text{sub}}}(F_{_m}(x)$ is finite morphism.
  Therefore, 
  \begin{equation}\label{p1}
  \text{dim}(Z)\leq \text{dim}(F_{_m}(x))-1
   \end{equation}
Note that for any $s\in H^0(X,\mathcal{O}_{_X}((m-n)H)\otimes m_{_x})$ the map $H^0(X,E(n))\to H^0(X,E(m)\otimes m_{_x})$ 
which sends $\phi\mapsto s\otimes \phi$ is injective. Thus for any $[s]\in \mathbb{P}_{_{\text{sub}}}H^0(X,\mathcal{O}_{_X}((m-n)H)$ $[s]\times K_{_m}(x)\subseteq Z$.
Therefore, $\pi_{_1}:Z\to \mathbb{P}_{_{\text{sub}}}H^0(X,\mathcal{O}_{_X}((m-n)H)$ is dominant where 
$\pi_{_1}:\mathbb{P}_{_{\text{sub}}}(H^0(X,\mathcal{O}_{_X}((m-n)H)\otimes m_{_x})))\times \mathbb{P}_{_{\text{sub}}}(H^0(X,E(n)))\to \mathbb{P}_{_{\text{sub}}}H^0(X,\mathcal{O}_{_X}((m-n)H)$
is the first projection. Consequently, we have 
\begin{equation}\label{p2}
 \text{dim}(Z)\geq h^0(\mathcal{O}((m-n)H\otimes m_{_x}))-1
\end{equation}
Combining \ref{p1} and \ref{p2} we get the Lemma.
 \end{proof}
\end{lemma}

As $F_{_m}\subsetneq M_{E(m)}$ and $\text{rank}(M_{E(m)})=h^{0}(E(m))-l$. Therefore, $\text{rank}(F_{_m})< h^{0}(E(m))-l$.
Note that by Riemann Roch formula and the above lemma, we have $\text{rank}(F_m) \ge \frac{m^2}{2} + O(m)$, where $O(m)$ is a linear function of $m$.
Thus we have, 
\begin{equation}\label{R1}
 \text{rank}(F_{_m})=f(m),~ where~ f(m)~ is~ a~ function~ of~ m~ such~ that~ \frac{f(m)}{m^2}=r,~ \frac{1}{2}\leq r\leq \frac{l}{2}, ~as~ m\to \infty
\end{equation}

\begin{lemma}\label{rk2}
 $\text{rank}(K_{_m})\geq r h^0(E(n))$ for large $m>>0$.
 \begin{proof}
  From the equation \ref{eqn6} we get 
  \begin{equation}\label{eqn7}
   \mu(F_{_m}|_{_{C_m}})=\frac{\text{deg}(K_{_m})+\text{deg}(N_{_m})}{\text{rank}(F_{_m})}\leq \frac{\text{deg}(N_{_m})}{\text{rank}(F_{_m})}=\mu(N_{_m})\frac{\text{rank}(N_{_m})}{\text{rank}(F_{_m})}.
  \end{equation}
Since $\overline{M_{_m}}$ is stable we get 
\begin{equation}\label{eqn7}
 \mu(N_{_m})\frac{\text{rank}(N_{_m})}{\text{rank}(F_{_m})}< \mu(\overline{M_{_m}})\frac{\text{rank}(N_{_m})}{\text{rank}(F_{_m})}=\mu(\overline{M_{_m}})(1-\frac{\text{rank}(K_{_m})}{\text{rank}(F_{_m})}).
\end{equation}
Now from \ref{eqn6} we have $\text{deg}(M_{_m}|_{_{C_{_m}}})=\text{deg}(\overline{M_{_m}})$, and since 
\[\mu(M_{_m}|_{_{C_m}})\leq \mu(F_{_m}|_{_{C_m}})\] from equation \ref{eqn7} we get 
\begin{equation}
 \frac{\text{deg}(M_{_m}|_{_{C_m}})}{h^0(E(n))+\text{rank}(\overline{M_{_m}})}<\frac{\text{deg}(M_{_m}|_{_{C_m}})}{\text{rank}(\overline{M_{_m}})}(1-\frac{\text{rank}(K_{_m})}{\text{rank}(F_{_m})}).
\end{equation}

Noting the fact $\text{deg}(M_{_m}|_{_{C_m}})<0$ we get 
\begin{equation}
 \frac{1}{h^0(E(n))+\text{rank}(\overline{M_{_m}})}>\frac{1}{\text{rank}(\overline{M_{_m}})}(1-\frac{\text{rank}(K_{_m})}{\text{rank}(F_{_m})})
\end{equation}
i.e.,
\[\frac{\text{rank}(\overline{M_{_m}})}{h^0(E(n))+\text{rank}(\overline{M_{_m}})}>(1-\frac{\text{rank}(K_{_m})}{\text{rank}(F_{_m})})\]

Thus 
\[\frac{\text{rank}(K_{_m})}{\text{rank}(F_{_m})}>1-\frac{\text{rank}(\overline{M_{_m}})}{h^0(E(n))+\text{rank}(\overline{M_{_m}})}=\frac{h^0(E(n))}{h^0(E(n))+\text{rank}(\overline{M_{_m}})}\]
Therefore,
\[\text{rank}(K_{_m})> h^0(E(n))\frac{\text{rank}(F_{_m})}{\text{rank}(M_{_m})}\]
By Lemma \ref{rk1} and since $\text{rank}(M_{_m})=h^0(E(m))-l$ we have 
\[\text{rank}(K_{_m})> h^0(E(n))\frac{rm^2 + q(m)}{h^0(E(m))-l}\]
We have 
$\frac{rm^2 + q(m)}{h^0(E(m))-l}=r-\epsilon(m)$ where $\epsilon(m)\to 0$ as $m\to \infty$.
Thus for large $m>>0$ we get the inequality 
\[\text{rank}(K_{_m})\geq r h^0(E(n)).\]
 \end{proof}
 \end{lemma}

{\bf Now we will proceed with the proof of Theorem $(1.1)$:}
Considering the fibres of the various vector bundles  appearing in the commutative diagram \ref{eqn6} we get 

\begin{center}
\begin{equation}\label{p4}
\begin{tikzcd}
K_{_m}(x)\arrow[r] \arrow[d] 
 & F_{_m}(x)\arrow[d] \\
 H^0(E(n))\arrow[r]
 & (M_{_m}|_{C_m})(x)=H^0(E(m)\otimes m_x)
 \end{tikzcd}
 \end{equation}
\end{center}

Let $G$ be the subsheaf of $E(n)$ generated by $K_{_m}(x)$. 
We claim that $\text{rank}(G)>1$.
Note that as $G$ is generated by $K_{_m}(x)$ we have $K_{_m}(x)\subseteq H^0(X,G)$ in other words 
\begin{equation}\label{p5}
 \text{rank}(K_{_m}(x))\leq h^0(G).
\end{equation}

As $E$ is rank $l$ slope stable vector bundle with respect to $H$, for any proper subsheaf $F\subset E$ i.e., $0<\text{rank}(F)<l$, we have , for $d>>0$,
\[\frac{\chi(F(d))}{rk(F(d))}<\frac{\chi(E(d))}{l}.\] where $\chi$ is the Euler characteristic of the respective sheaves.
Considering $F:=G(-n)$ we get, from the above equation,
\[\frac{\chi(G)}{rk(G)}< \frac{\chi(E(n))}{l}\] Now we choose $n$, sufficiently large, such that $\chi(G)=h^0(G)$ and $\chi(E(n))=h^0(E(n))$.
Therefore, we have 
\begin{equation}\label{p6}
 \frac{h^0(G)}{rk(G)}<\frac{h^0(E(n))}{l}.
\end{equation}

Combining equation \ref{p5} and \ref{p6} we get $\frac{\text{rank}(K_{_m}(x))}{rk(G)}< \frac{h^0(E(n))}{l}$. On the other hand 
by Lemma \ref{rk2} $\frac{rh^0(E(n))}{rk(G)} \leq \frac{\text{rank}(K_{_m}(x))}{rk(G)}$ where $\frac{1}{2}\leq r\leq \frac{l}{2}$. 
This immediately implies that $rk(G)>lr\geq 1$ since $r\geq 
\frac{1}{2}$ and $l\geq 2$. 
We have $h^0(G(m-n)\otimes m_{_x})=h^0(G(m-n))-rk(G)=\frac{rk(G)}{2}m^2+p(m)$ where $p\in \mathbb{Q}[n][m]$ and linear in $m$. Let 
$t=\frac{rk(G)}{2}$ then $t\geq 1$.
Now we consider two cases separately.\\
{\bf Case I }: $r <1$.\\
As $t\geq 1$ by 
equation \ref{R1}  for large $m>>0$ 
\begin{equation}\label{R2}
\text{rank}(F_{_m})< h^0(G(m-n)\otimes m_{_x}).
 \end{equation}

For any $x\in X$ from the inclusion $G\hookrightarrow E(n)$ we obtain $H^0(G((m-n)\otimes m_{_x})\hookrightarrow H^0(E(m)\otimes m_{_x})$ 
and by \ref{R2} we get $\text{Im}(F_{_m}(x))\neq \text{Im}(H^0(G((m-n)\otimes m_{_x}))$ inside $H^0(E(m)\otimes m_{_x})$.
Let $B_{_m}(x)=\text{Im}(F_{_m}(x))\cap \text{Im}(H^0(G((m-n)\otimes m_{_x}))$. Then $B_{_m}(x)\subsetneq H^0(G((m-n)\otimes m_{_x}))$
For any section $s\in H^0(\mathcal{O}_{_X}((m-n))\otimes m_{_x})$ multiplication by $s$ maps 
\[H^0(G)\to H^0(G(m-n)\otimes m_{_x})\] and \[K_{_m}(x)\to B_{_m}(x).\] Thus from the  commutative diagram \ref{p4}
we get 
\begin{center}
\begin{equation}
\begin{tikzcd}
K_{_m}(x)\arrow[r] \arrow[d] 
 & B_{_m}(x)\arrow[d] \\
 H^0(G)\arrow[r]
 & H^0(G(m-n)\otimes m_x)
 \end{tikzcd}
 \end{equation}
\end{center}

Since $K_{_m}(x)$ generates $G$ by $(2)$ we conclude that 
\[H^0(\mathcal{O}_{_X}((m-n))\otimes m_{_x})\otimes K_{_m}(x) \to H^0(G(m-n)\otimes m_{_x})\] is surjective. This leads to a contradiction 
as we vary sections over an open set of $H^0(\mathcal{O}((m-n)H))$ the images of $K_{_m}(x)$ span the whole vector space $H^0(G(m-n)\otimes m_{_x})$.
But from the above commutative diagram every image lies on the proper fixed subspace $B_{_m}(x)$.
Therefore, $E(m)$ is $H$-stable.\\
{\bf Case II}: $r \geq 1$.\\
In this case by Lemma \ref{rk2} rank of $K_m = h^0(E(n))$. In other words, $G = E(n)$. Thus the Theorem follows by similar arguments as in Case I.

\end{document}